\documentclass[table, svgnames, dvipsnames, letterpaper, 11pt]{report}

\usepackage{amsmath, amsthm, latexsym, amssymb, graphicx, graphics, bold-extra, multirow,
pdfpages, array, float, subcaption, bm, changepage, longtable, booktabs, algorithmic, wrapfig}

\usepackage{xcolor, soul}

\usepackage{mathrsfs}
\usepackage{stmaryrd}

\usepackage{titlesec}
\titlespacing\section{0pt}{12pt plus 4pt minus 2pt}{0pt plus 2pt minus 2pt}
\titlespacing\subsection{0pt}{12pt plus 4pt minus 2pt}{0pt plus 2pt minus 2pt}
\titlespacing\subsubsection{0pt}{12pt plus 4pt minus 2pt}{0pt plus 2pt minus 2pt}

\usepackage{etoolbox} % provides \patchcmd
\patchcmd{\endabstract}{\null}{}{}{}
\patchcmd{\thebibliography}{\chapter*}{\section*}{}{}
\makeatletter
\patchcmd{\@tocline}{\hfill}{%
  \leaders\hbox{$\m@th
    \mkern \@dotsep mu\hbox{.}\mkern \@dotsep
  mu$}\hfill}{}{}
\makeatother

\makeatletter
\patchcmd{\l@section}{\hfil}{%
  \leaders\hbox{$\m@th
    \mkern \@dotsep mu\hbox{\color{blue}.}\mkern \@dotsep
  mu$}\hfill}{}{}
\makeatother


\usepackage{tocloft}
\usepackage[T1]{fontenc}
\usepackage{lmodern}
\usepackage{textcomp}

\usepackage{arydshln}

\usepackage[sort&compress, square, numbers]{natbib}

\usepackage[resetlabels]{multibib}

\usepackage{tikz} 
\usetikzlibrary{positioning}

\usepackage{geometry}
\theoremstyle{remark}
\newtheorem{theorem}{Theorem}
\newtheorem{remark}[theorem]{Remark}
\newtheorem{definition}[theorem]{Definition}

\usepackage{fancyhdr}
\usepackage[pdftex, breaklinks=true]{hyperref}
\hypersetup{
    unicode=false,                              % non-Latin characters in Acrobat's bookmarks
    pdftoolbar=true,                            % show Acrobat's toolbar?
    pdfmenubar=true,                            % show Acrobat's menu?
    pdffitwindow=true,                          % page fit to window when opened
    pdftitle={ACMS Jorunal and Proceedings},    % title
    pdfsubject={Subject},                       % subject of the document
    pdfnewwindow=true,                          % links in new window
    pdfkeywords={keywords},                     % list of keywords
    colorlinks=true,                            % false: boxed links; true: colored links
    linkcolor=blue,                             % color of internal links
    citecolor=blue,                             % color of links to bibliography
    filecolor=magenta,                          % color of file links
    urlcolor=blue,                              % color of external links
    linktoc=all
}

\makeatletter
\renewenvironment{abstract}{%
      \@beginparpenalty\@lowpenalty
      \small
      \begin{center}%
        \bfseries \abstractname
        \@endparpenalty\@M
        \vspace*{-0.2in}
      \end{center}\quotation}%
     {\endquotation\par%
     }
\makeatother
\begin{document}

% Here is where your paper begins

\begin{center}
{\bf \LARGE Self-graphing equations}\\

\medskip

% Give your personal website (if any), and institution website below
% Author Website \, (Institution Website)

% Change ``authorlastname'' to your last name in the field below.
\label{Alexander}

% Here is where a photo of you will appear in the journal.
Samuel Allen Alexander
\end{center}

\vspace*{-0.2in}

\begin{abstract}
\noindent
Can you find an $xy$-equation that, when graphed, writes itself on the plane?
This idea became internet-famous when a Wikipedia article on Tupper’s self-referential formula went viral in 2012. Under scrutiny, the question has two flaws: it is meaningless (it depends on typography) and it is trivial (for reasons we will explain). We fix these flaws by formalizing the problem, and we give a very general solution using techniques from computability theory.
\end{abstract}

\section{Introduction}

Suppose your friend sends you an $xy$-equation and you start graphing it.
After graphing for a few minutes, you notice that what you've graphed so far
looks like the letter $x$. You continue graphing, and you notice that you've
just plotted the letter $y$ on your graphing paper.
After some more work, you notice that you've
just added the symbol $+$ in between the $x$ and the $y$.
You continue this way for many hours until the equation is completely graphed.
Then you step back and realize that what you've written on your graphing paper
is the very equation your friend sent you!

A self-graphing equation is an equation such that, when you graph it, you get
the equation itself back, written on your graphing paper. This idea received
a lot of attention when a Wikipedia article on Tupper's self-referential
formula went viral in 2012\footnote{Tupper's so-called self-referential formula
is not actually self-referential at all (nor did he himself call it
self-referential \cite{tupper2001reliable}). Rather, it's a formula whose graph
contains every possible $106\times 17$-pixel bitmap. Tupper later posted an
actually self-referential formula on his website \cite{tupperwebpage},
but it received less attention.
Tupper's original formula has been generalized by Somu
and Mishra \cite{somu2023generalization}.
Tr\'avn\'ik has also published
a self-graphing formula \cite{travnikwebpage}.}.

The problem of finding a self-graphing equation is meaningless because
it depends on typography. It's also trivial, if no limit is placed on
what functions one can use in the equation and how they are written.
Indeed, fix a particular image $I\subseteq\mathbb R^2$ of the equation
``$F(x,y)=1$'' on the plane (for example, the letter $F$ could be the union
of a vertical line segment and two horizontal line segments; the
parentheses could be fragments of B\'ezier curves; and so on),
and define $F:\mathbb R\to\{0,1\}$ by
\[
  F(x,y)
  =
  \begin{cases}
    1 &\mbox{if $(x,y)\in I$,}\\
    0 &\mbox{otherwise.}
  \end{cases}
\]
By construction the graph of the equation $F(x,y)=1$ is $I$, an image
of the equation $F(x,y)=1$ on the plane. So $F(x,y)=1$ is trivially a
self-graphing equation. Clearly, in order to make the problem nontrivial,
it is necessary to specify which functions are allowed, and how they are
written!

We will address the two problems of meaninglessness and triviality by
formalizing the problem. Then, rather than focusing on any particular
typography or any particular choice of what functions are allowed,
we will instead give sufficient conditions thereon. Any typography, and
any choice of functions, which satisfies these sufficient conditions,
will be guaranteed to yield a self-graphing equation. To do this, we will
invoke the so-called \emph{recursion theorem} from computability theory
(appropriately, the same theorem which was classically used to prove the
existence of self-printing computer programs, also known as \emph{Quines}).

\section{Formalization}

\begin{quote}
  ``What was a compelling proof in 1810 may well not be now;
  what is a fine closed form in 2010 may have been anathema a century
  ago'' \cite{borwein2013closed}
\end{quote}

\begin{definition}
  If $A\not=\emptyset$ is a finite
  alphabet, write $A^*$ for
  the set of finite strings
  from $A$. By a \emph{notion
  of equations} we mean a finite
  alphabet $A$ together with a
  function
  $\mathrm{Gr}:A^*
  \to\mathscr P(\mathbb R^2)$
  assigning to every string
  $\sigma\in A^*$ a subset
  $\mathrm{Gr}(\sigma)$ of $\mathbb R^2$
  called the \emph{graph} of $\sigma$.
\end{definition}

For example, if $A$ contains symbols $x$, $y$, $+$, ${}^2$, $=$ and $1$, and
if $\sigma\in A^*$ is the string ``$x^2+y^2=1$'', then $\mathrm{Gr}(\sigma)$
might be (but we do not require it to be!) the unit circle centered at the
origin. Or, if $A$ contains symbols $r$, $\theta$, $\cos$,
$=$ and $1$, and if $\sigma$ is the string ``$r=1+\cos\theta$'', then
$\mathrm{Gr}(\sigma)$ might be the graph
of a cardioid. Or if $\sigma$ is the string ``$+=$''
(or if $\sigma$ is the blank string), then $\mathrm{Gr}(\sigma)$
might be an error message, ``Error: Invalid
equation'', written on the plane (as, say, a union of points, line segments,
and B\'ezier curve fragments).

\begin{definition}
  By a \emph{glyphed notion of equations} we mean a triple
  $(A,\mathrm{Gr},\mathrm{Gl})$ where $(A,\mathrm{Gr})$ is a notion
  of equations and $\mathrm{Gl}:A\to\mathscr{P}(\mathbb R^2)$
  is a function assigning to each $x\in A$ a set
  $\mathrm{Gl}(x)\subseteq\mathbb R^2$ called the \emph{glyph} of $x$.
\end{definition}

If $A$ contains the symbol $0$, then $\mathrm{Gl}(0)$ might be,
for example, the circle of radius $\frac12$ centered at
$(\frac12,\frac12)$ (so as to nicely fit in the $1\times 1$ unit
square $[0,1]^2$, lending itself to a monospace font where each character
is $1$ unit wide). But it does not have to be. If $A$ contains the symbol $X$,
then $\mathrm{Gl}(X)$ might be, for example, the union of the line segment from
$(0,0)$ to $(1,1)$ and the line segment from $(0,1)$ to $(1,0)$ (again nicely
lending itself to a monospace font where each character is $1$ unit wide). But
it does not have to be.

\begin{definition}
\label{extendingglyphsdefn}
  (Extending glyphs to strings)
  \begin{enumerate}
    \item
    For all $S\subseteq\mathbb R^2$ and all $r\in\mathbb R$,
    let $S^{\rightarrow r}=\{(x+r,y)\,:\,(x,y)\in S\}$, the result
    of translating $S$ to the right by $r$ units.
    \item
    Whenever $(A,\mathrm{Gr},\mathrm{Gl})$ is a glyphed notion of equations,
    we will extend $\mathrm{Gl}$ to a function on $A^*$,
    also written $\mathrm{Gl}$ (this will cause no confusion),
    as follows. Let $\sigma\in A^*$.
    \begin{itemize}
      \item If $\sigma$ is the empty string, let $\mathrm{Gl}(\sigma)=\emptyset$.
      \item If $\sigma$ is the string of length $1$, whose first (and only)
        character is $x\in A$, let $\mathrm{Gl}(\sigma)=\mathrm{Gl}(x)$.
      \item Otherwise, $\sigma$ is the string $x_0\ldots x_k$ where each
        $x_i\in A$. Let
        \[
          \mathrm{Gl}(\sigma)
          =
          \mathrm{Gl}(x_0)^{\rightarrow 0}
          \cup
          \cdots
          \cup
          \mathrm{Gl}(x_k)^{\rightarrow k}.
        \]
    \end{itemize}
  \end{enumerate}
\end{definition}

Thus, $\mathrm{Gl}(\sigma)$ is the result of writing $\sigma$ on the plane,
from left to right, translating the glyph of each $i$th character to the right
by $i$ units. The resulting union is particularly easy to visualize if we
assume that for every $x\in A$, $\mathrm{Gl}(x)\subseteq [0,1]^2$. In that case,
the glyphs of $A$ comprise a monospace font where every character has width $1$,
and the glyph of a string in $A^*$ is the result of writing the glyphs of the
individual characters from left to right in the usual way. This assumption will
make the results in this paper more intuitive, but, interestingly, the whole
paper will work just fine without this assumption.

\begin{definition}
  (Self-graphing equations)
  Let $\mathcal A=(A,\mathrm{Gr},\mathrm{Gl})$ be a glyphed notion of equations.
  By a \emph{self-graphing equation in $\mathcal A$} we mean a string
  $\sigma\in A^*$ such that $\mathrm{Gr}(\sigma)=\mathrm{Gl}(\sigma)$.
\end{definition}

\section{Computability theory preliminaries}

\begin{definition}
  \begin{enumerate}
    \item
    For any sets $X$ and $Y$, we write $f:{\subseteq}X\to Y$ to indicate
    that $f$ is a function whose codomain is $Y$ and whose domain is
    some subset of $X$.
    \item
    For all $n\in\mathbb N$,
    let $\varphi_n:{\subseteq}\mathbb N\to \mathbb N$
    be the $n$th computable function
    (assuming some fixed enumeration, possibly with repetition,
    of the computable functions).
    \item
    A function $f:{\subseteq}\mathbb N\to\mathbb N$ is \emph{total computable}
    if $\mathrm{dom}(f)$ (the domain of $f$) is all of $\mathbb N$.
  \end{enumerate}
\end{definition}

We state the following celebrated result from computability theory
without proof.

\begin{theorem}
\label{recursionthm}
  (The Recursion Theorem)
  For every total computable $f:\mathbb N\to\mathbb N$, there is some
  $n\in\mathbb N$ such that $\varphi_n=\varphi_{f(n)}$.
\end{theorem}

\section{Self-constraint: a sufficient condition for the existence of self-graphing
equations}

In this section we fix a glyphed notion of equations
$\mathcal A=(A,\mathrm{Gr},\mathrm{Gl})$ (where $A$ is a finite nonempty
alphabet).

\begin{definition}
  By a \emph{G\"odel numbering of $A^*$} we mean a bijection\footnote{One could
  change this definition to require only that $\ulcorner\bullet\urcorner$ be
  an injection instead of a bijection, which would be more typical of
  G\"odel numberings. We chose to require the G\"odel numbering function to
  be bijective in order to avoid technical complications.}
  $\ulcorner\bullet\urcorner:A^*\to\mathbb N$ such that
  there is some algorithm for computing $\ulcorner\sigma\urcorner$
  (for $\sigma\in A^*$)
  as a function of $\sigma$.
  We refer to $\ulcorner\sigma\urcorner$ as the \emph{G\"odel number}
  of $\sigma$ (we think of $\ulcorner\sigma\urcorner$ as a numerical
  encoding of $\sigma$).
\end{definition}

\begin{definition}
\label{selfconstraintdefn}
  The glyphed notion of equations $\mathcal A$ is \emph{self-constrained}
  if there exists a G\"odel numbering $\ulcorner\bullet\urcorner$
  of $A^*$ and a total computable $f:\mathbb N\to\mathbb N$ such that:
  \begin{itemize}
    \item
    For all $n\in\mathbb N$, if $\varphi_n(0)=\ulcorner\tau\urcorner$
    for some $\tau\in A^*$, then
    $f(n)=\ulcorner\sigma\urcorner$ for some $\sigma\in A^*$
    such that $\mathrm{Gr}(\sigma)=\mathrm{Gl}(\tau)$.
  \end{itemize}
\end{definition}

If $f$ is as in Definition \ref{selfconstraintdefn}, then $f$ should intuitively be
thought of as being computed by an algorithm which takes an input $n\in\mathbb N$
and outputs an equation whose graph is the output of
$\varphi_n(0)$ (if any), written on the plane. The strings in question are
encoded by G\"odel numbers to standardize the functions in question and allow
the usage of standard computability theory, but intuitively one should think of
$f$ and $\varphi_n$ as outputting strings from $A^*$.
If $0\not\in\mathrm{dom}(\varphi_n)$ then it does not matter what $f(n)$ is,
only that $f(n)$ be defined.

\begin{remark}
\label{importantremark}
It is not required, in the algorithm which
computes $f(n)$, for $\varphi_n(0)$ to actually be computed as a preliminary
step. It is not even required that the algorithm computing $f(n)$ determine
whether or not $\varphi_n(0)$ exists (and indeed, this would be impossible,
as it would require solving the Halting Problem). The work of computing
$\varphi_n(0)$, or even of determining whether $\varphi_n(0)$ exists, can
be delegated to whoever has to \emph{graph} the output of $f(n)$.
\end{remark}

We can illustrate Remark \ref{importantremark} with the
following analogy. Say that $k\in\mathbb N$ is an \emph{FLT-counterexample}
(here FLT stands for ``Fermat's Last Theorem'')
if $k>2$ and there exist positive integers $a,b,c$ such that $a^k+b^k=c^k$.
For every $x\in\mathbb R$, let $\psi(x)$ be the number of FLT-counterexamples
$\leq x$. A teacher does not need to know Fermat's Last Theorem in order to
assign a student the task of graphing the equation $y=\psi(x)$.
Without knowing Fermat's Last Theorem is true, a teacher can even, with
some tedious mechanical effort, rewrite $y=\psi(x)$ in ``closed form''
(at least if the closed form is allowed to include infinite sums---see
\cite{alexander2006formulas}). Knowledge of Fermat's Last Theorem is
required in order to \emph{graph} the equation, not to \emph{state} it.

We will now show that self-contraint is a sufficient condition for existence
of a self-graphing equation. At first glance, self-constraint might seem like
such a strong requirement as to leave one in doubt whether any reasonable
notions of equations actually satisfy it. We will give an example in
Section \ref{concretesection} of a notion of equations which is self-constrained
and therefore has a self-graphing equation, and the example should help the
reader to better understand how self-constraint can be satisfied. Basically,
the key is that infinite products or infinite sums can be used to encode
quantifiers $\exists$ and $\forall$.

\begin{theorem}
\label{maintheorem}
  If $\mathcal A$ is self-constrained then there exists a self-graphing
  equation in $\mathcal A$.
\end{theorem}

\begin{proof}
  Let $\ulcorner\bullet\urcorner$ and $f:\mathbb N\to\mathbb N$ be as in
  Definition \ref{selfconstraintdefn}.

  Subclaim: There is a total computable function $g:\mathbb N\to\mathbb N$
  such that for all $n\in\mathbb N$, $\varphi_{g(n)}(0)=f(n)$.

  This Subclaim is actually a special case of a theorem
  from computability theory called
  the ``$Smn$ theorem'', but we will sketch a direct proof here.
  Let $g:\mathbb N\to\mathbb N$ be the function computed by the following
  algorithm:
  \begin{enumerate}
    \item Take input $n\in\mathbb N$.
    \item Let $X=f(n)$.
    \item Let $P$ be the following algorithm:
    \begin{enumerate}
      \item Take input $m\in\mathbb N$.
      \item Output $X$ (ignoring the value of $m$).
    \end{enumerate}
    \item Output an encoding of $P$ (a number $k$ such that
      $\varphi_k$ is the function computed by $P$).
  \end{enumerate}
  For any $n\in\mathbb N$, by construction $\varphi_{g(n)}$ is the
  function computed by the above algorithm $P$ (for the given $n$).
  Thus $\varphi_{g(n)}(0)$ is computed by ignoring the input $m=0$
  and outputting $X=f(n)$. Thus $\varphi_{g(n)}(0)=f(n)$.
  Since we have provided an algorithm for $g$, $g$ is computable.
  Clearly $\mathrm{dom}(g)=\mathbb N$, so $g$ is total computable.
  This proves the Subclaim.

  Let $g$ be as in the Subclaim.
  By the Recursion Theorem (Theorem \ref{recursionthm}) there is some
  $n\in\mathbb N$ such that $\varphi_n=\varphi_{g(n)}$.
  In particular,
  \[
    \varphi_n(0) = \varphi_{g(n)}(0) = f(n) \mbox{ is defined. ($*$)}
  \]
  Let $\sigma,\tau\in A^*$ be such that $f(n)=\ulcorner\sigma\urcorner$
  and $\varphi_n(0)=\ulcorner\tau\urcorner$.
  We claim $\sigma$ is a self-graphing equation in $\mathcal A$.
  To see this,
  compute:
  \begin{align*}
    \mathrm{Gr}(\sigma)
      &= \mathrm{Gl}(\tau)
        &\mbox{(Definition \ref{selfconstraintdefn})}\\
      &= \mathrm{Gl}(\sigma),
        &\mbox{(By $*$, $\sigma=\tau$)}
  \end{align*}
  as desired.
\end{proof}

\section{A Concrete Context for a Self-Graphing Equation}
\label{concretesection}

To conclude, we will give an example of a
particular glyphed notion of equations
$\mathcal A=(A,\mathrm{Gr},\mathrm{Gl})$ not too unlike how we
write and graph equations in practice. We will argue that this particular
$\mathcal A$ is
self-constrained. Thus, Theorem \ref{maintheorem} guarantees the
existence of a self-graphing equation in $\mathcal A$.

For an alphabet, let
\begin{align*}
  A={}&
    \{a,b,c,\ldots,z\}
      &\mbox{(Letters)}\\
  &{} \cup\{0,1,2,\ldots,9\}
      &\mbox{(Digits)}\\
  &{} \cup\{(\}\cup\{)\}
      &\mbox{(Left and right parentheses)}\\
  &{} \cup\{+,\cdot,-,/,{}^\wedge,=\}
      &\mbox{(Plus, times, minus, division, exponentiation, equality)}\\
  &{} \cup\{\Pi,\_,\infty\}
      &\mbox{(Infinite product machinery)}
\end{align*}
(for concreteness, $A$ can be taken to be a subset of $\mathbb N$ of
cardinality $26+10+2+6+3=47$).
The reader should think of ${}^\wedge$ as an exponentiation operator,
as in the equation $2{}^\wedge 3=8$ (read: ``$2$ to the power $3$ equals $8$'').
The character $\Pi$ should be thought of
as an infinite product symbol, to be used (in combination with ${}^\wedge$,
$\_$, $=$, $\infty$, and parentheses) as in the equation:
$\Pi\_(n=0){}^\wedge\infty(1^\wedge n)=1$ (read: ``The product, as $n$ goes from
$0$ to $\infty$, of $1^n$, equals $1$'').

Choose glyphs $\mathrm{Gl}:A\to\mathcal{P}(\mathbb R^2)$
for writing $A$ such that each such glyph is
written inside the square $[0,1]\times[0,1]$ using pixels of dimension
$\frac1{100}\times\frac1{100}$, each such pixel being a translation, by an integer
multiple of $\frac1{100}$ horizontally and an integer multiple of $\frac1{100}$
vertically, of the square $[0,\frac1{100}]\times[0,\frac1{100}]$.
For example, $\mathrm{Gl}(+)$, the glyph of the $+$ sign, might
be
$([\frac{50}{100},\frac{51}{100}]\times[0,1])
\cup
([0,1]\times[\frac{50}{100},\frac{51}{100}])$
(the first argument to $\cup$ being a rectangle of height $1$ and width $1/100$
and the second argument to $\cup$ being a rectangle of height $1/100$ and
width $1$), which can clearly be formed by such pixels.

Define $\mathrm{Gr}:A^*\to\mathscr P(\mathbb R^2)$ so that for every
$\sigma\in A^*$, if $\sigma$ is a valid equation, then $\mathrm{Gr}(\sigma)$
is the graph of $\sigma$. If $\sigma$ is not a valid equation, then let
$\mathrm{Gr}$ be some arbitrary nonempty subset of $\mathbb R^2$, for example,
an error message written on the plane (we only require it to be nonempty so
as not to inadvertently make the empty string a trivial self-graphing equation).
For example, if $\sigma$ is the string ``$x{}^\wedge2+y{}^\wedge2=1$'', then
$\mathrm{Gr}(\sigma)$ is the unit circle; if $\sigma$ is the string
``$x{}^\wedge 2=-1$'' then $\mathrm{Gr}(\sigma)$ is the empty set.

In this way, we obtain a glyphed notion of equations
$\mathcal A=(A,\mathrm{Gr},\mathrm{Gl})$. We will argue that $\mathcal A$
is self-constrained and thus (by Theorem \ref{maintheorem})
admits a self-graphing equation. In other words, we will argue
(Definition \ref{selfconstraintdefn}) that there is a G\"odel
numbering $\ulcorner\bullet\urcorner$ of $A^*$ and a total computable
$f:\mathbb N\to\mathbb N$ such that for all $n\in\mathbb N$, if
$\varphi_n(0)=\ulcorner\tau\urcorner$
then $f(n)=\ulcorner\sigma\urcorner$ for some $\sigma\in A^*$ such that
$\mathrm{Gr}(\sigma)=\mathrm{Gl}(\tau)$.

Let $\ulcorner\bullet\urcorner:A^*\to \mathbb N$ assign numbers bijectively
to strings from $A^*$ in some way that could be written out as an algorithm.
There are many ways to do this and it does not matter which way it is done.
As one example, we could linearly order $A$ and then enumerate $A^*$ by
listing all the length-$0$ strings in $A^*$ (in alphabetical order), followed
by all the length-$1$ strings in $A^*$ (in alphabetical order),
followed by all the length-$2$ strings in
$A^*$ (in alphabetical order) and so on,
and let each $\ulcorner \sigma\urcorner$ be the position
in which $\sigma$ occurs in the resulting list.

We want $f(n)$ to output $\ulcorner\sigma\urcorner$ for some $\sigma\in A^*$
such that the graph of $\sigma$ is $\mathrm{Gl}(\tau)$,
where $\tau\in A^*$ is the string
whose code is output by $\varphi_n(0)$ (if $0\in \mathrm{dom}(\varphi_n)$).
For such $\tau$, what does
it mean for a pair $(x,y)\in\mathbb R^2$ to be in $\mathrm{Gl}(\tau)$?
It means that...
\begin{equation}
  \tag{*}
  \exists a,b,c,d,e \in\mathbb N
  \mbox{ s.t.\ } P(n,a,b,c,d,e)
\end{equation}
...where $P(n,a,b,c,d,e)$ is the statement:
\begin{quote}
  The $n$th Turing machine (i.e., the Turing machine which
  computes $\varphi_n$), when run with input $0$, halts after exactly
  $a$ steps, with output $b$, and when $b$ is interpreted as
  a string $\tau$ (using $\ulcorner\bullet\urcorner$), $\tau$ has length at
  least $c+1$---let the $c$th symbol of $\tau$ (counting from $0$) be
  called $\tau_{c}$---and the $\frac1{100}\times\frac1{100}$ pixel
  with bottom-left coordinates $(d/100,e/100)$ is an element of
  $\mathrm{Gl}(\tau_{c})$, and $(x-c,y)$ (the result of translating $(x,y)$
  to the left by $c$ units) is in said pixel (so that $(x,y)$ is in the
  translation of said pixel by $c$ units to the right, which is said pixel's
  representation in $\mathrm{Gl}(\tau)$ by Definition \ref{extendingglyphsdefn}).
\end{quote}

Let's examine the subclauses of ($*$).
\begin{itemize}
  \item
  The subclause ``The $n$th Turing machine, when run
  on input $0$, halts after exactly $a$ steps, with output $b$'', can be
  expanded out into a complicated statement in the language of arithmetic
  (``There exists $k$ such that $k$ encodes a sequence $C_0,C_1,\ldots,C_a$
  of Turing machine snapshots such that...'').
  \item
  The subclause
  ``the $\frac1{100}\times\frac1{100}$ pixel
  with bottom-left coordinates $(d/100,e/100)$ is an element of
  $\mathrm{Gl}(\tau_{c})$'', can be written as a finite disjunction
  of quantifier-free statements about individual pixels, namely,
  at most $100\cdot 100\cdot |A|$ such disjuncts: one per
  $\frac1{100}\times\frac1{100}$ pixel in $[0,1]^2$ per
  symbol in $A$. For example, if the glyph of symbol ``$+$'' includes
  pixel $[50/100,51/100]\times [0,1/100]$, then this pixel-symbol pair
  contributes the
  quantifier-free disjunct:
  $(\tau_{c}=\mbox{``x''})\wedge (d=50)\wedge (e=0)$.
  \item
  The subclause ``$(x-c,y)$ is in said pixel'' can be
  rephrased as ``$d/100\leq x-c\leq (d+1)/100$ and
  $e/100\leq y\leq (e+1)/100$''.
\end{itemize}

We claim that all subclauses of ($*$) can be written as equations
of the form $E=0$ using only symbols from $A$; to see this, we reason
inductively:
\begin{itemize}
  \item
  Atomic subclauses like ``$d=50$'' can be written
  as $d-50=0$.
  \item
  Atomic subclauses like ``$e/100\leq y$'' can be rewritten as
  ``$y-e/100-|y-e/100|=0$'', and the absolute values can be
  replaced with symbols from $A$ by using the fact that
  $|x|=(x^2)^{1/2}$.
  \item
  (Disjunction)
  If two subclauses can be written in the form $E_1=0$ and $E_2=0$
  using only symbols from $A$, then so can their disjunction, because
  ``$E_1=0$ or $E_2=0$'' is equivalent to ``$(E_1)\cdot (E_2)=0$''.
  \item
  (Negation)
  If a subclause can be written in the form $E=0$ using only symbols
  from $A$, then so can its negation, because ``$\mbox{not}(E=0)$''
  is equivalent to $0^{E^2}=0$ (since $0^0=1$ \cite{knuth1992two}
  but $0^x=0$ for all positive
  $x$).
  \item
  (Existential Quantifiers)
  If a subclause $E=0$ can be written using only symbols from $A$
  (where $E$ may involve a variable $v$),
  then so can the clause $\exists v (E=0)$ for any variable $v$,
  because $\exists v (E=0)$ is equivalent to
  $\Pi_{v=0}^\infty (1-0^{E^2})=0$.
  \item
  (Conjunction, Universal Quantifiers)
  Closure under conjunction and universal quantification follow because
  ``$E_1=0$ and $E_2=0$'' is equivalent to
  ``$\mbox{not}(\mbox{not}(E_1=0)\mbox{ or }\mbox{not}(E_2=0))$''
  and ``$\forall v (E=0)$'' is equivalent to
  ``$\mbox{not}(\exists v \mbox{ not}(E=0))$''.
\end{itemize}

Thus, ($*$) itself can be written as an equation $E=0$ using only symbols
from $A$. Fix such an $E$.
For every $n\in\mathbb N$, let $\overline{n}$ be the string of $n$'s
decimal digits (for example if $n=311$ then $\overline{n}$ is the length-$3$
string ``311''). For every $n\in\mathbb N$, let $E(\overline n)=0$ be the equation
obtained by replacing all unquantified occurrences of $n$ in $E=0$
by $\overline{n}$. Define $f:\mathbb N\to\mathbb N$ so that
for all $n\in\mathbb N$, $f(n)=\ulcorner E(\overline n)=0\urcorner$.

By construction, $f(n)$ outputs (the code of) the equation $E(\overline n)=0$
whose graph is the set of all points $(x,y)$ satisfying ($*$),
i.e., the set of all points in $\mathrm{Gl}(\tau)$ where
$\ulcorner\tau\urcorner=\varphi_n(0)$ if such a $\tau$ exists.

Thus, $f$ witnesses that $\mathcal A$ is self-constrained.
By Theorem \ref{maintheorem}, there is a self-graphing equation in $\mathcal A$.

In some sense, the crucial key in this example is that the infinite product
allows for the expression of the unbounded logical quantifier $\exists$.
Together with the propositional logical connectives (AND, OR, NOT),
unbounded quantification enables expression of anything that can be expressed
in first-order logic.

% For bibliographic items, it is suggested that, to avoid reference clashes with
% other authors, you append the initials of your first and last name to the
% citation tags. In what follows, suppose your name is Linda Johnson, and that
% the first two citations are to articles by Donald Knuth and Bob Brabenec.
% Then the bibliography would look like
\setbiblabelwidth{10}
\bibliographystyle{plain}
\bibliography{Alexander_ACMS2024}

\end{document}